\documentclass[11pt, a4paper, reqno]{amsart}

\usepackage[margin=1in]{geometry}
\usepackage{amsmath, amssymb, amsthm}
\usepackage{microtype}
\usepackage{pgfplots}
\pgfplotsset{compat=1.18}
\usepackage[
colorlinks=true,
urlcolor=blue,
linkcolor=blue,
citecolor=blue,
pdfstartview=FitH
]{hyperref}

\newtheorem{theorem}{Theorem}[section]
\newtheorem{lemma}[theorem]{Lemma}
\newtheorem{proposition}[theorem]{Proposition}
\newtheorem{corollary}[theorem]{Corollary}

\theoremstyle{definition}
\newtheorem{definition}[theorem]{Definition}
\newtheorem{example}[theorem]{Example}

\theoremstyle{definition}

\theoremstyle{remark}
\newtheorem{remark}[theorem]{Remark}

\theoremstyle{plain}
\newtheorem{assumption}{Assumption}

\numberwithin{equation}{section}

\newcommand{\R}{\mathbb R}
\newcommand{\C}{\mathbb C}
\newcommand{\Sph}{\mathbb S}
\newcommand{\D}{\mathcal D}
\newcommand{\E}{\mathcal E}

\DeclareMathOperator{\Div}{div}
\DeclareMathOperator{\Span}{span}

\title[Higher order isoperimetric inequalities]
{Higher order isoperimetric inequalities and rigidity
for surfaces in $\R^n$}

\author[X.-Q. Fan]{Xu-Qian Fan}
\address{
School of Mathematics,
Jinan University,
Guangzhou 510632, China
}
\email{txqfan@jnu.edu.cn}

\author[K.-K. Kwong]{Kwok-Kun Kwong}
\address{
School of Mathematics and Physics,
University of Wollongong,
Wollongong, NSW 2522, Australia
}
\email{kwongk@uow.edu.au}

\dedicatory{Dedicated to the memory of Professor Luen-Fai Tam} 

\begin{document}
\begin{abstract}
We establish a family of sharp higher order isoperimetric type inequalities for smooth compact surfaces in Euclidean spaces of arbitrary codimension. These inequalities involve arbitrarily high derivatives of the boundary curve, together with the mean curvature vector of the surface. For example, the first and second order inequalities yield the sharp two-sided estimate for the isoperimetric deficit:
\begin{align*}
& \frac{\pi}{R}\int_{\partial M}|x-R\nu|^2\, ds +2\pi\int_M x\cdot H\, dA
\le L^2-4\pi A\\
\le& \frac{\pi}{R}\int_{\partial M}|x-R\nu|^2\, ds +2\pi\int_M x\cdot H\, dA
+\frac{\pi}{3R}\int_{\partial M}|x+R^2\kappa|^2\, ds,
\end{align*}
where $R=\frac{L}{2\pi}$ and the origin is chosen to be the centroid of the boundary. For minimal surfaces, equality in the lower bound holds precisely for flat disks, whereas the second inequality admits non-round equality cases. Our results extend that of Carleman, Sachs, Chakerian, and Li-Schoen-Yau.

A distinctive feature of our work is the unexpected rigidity structure arising at higher orders. At the first and second orders, equality curves admit decompositions into circular Fourier modes supported in mutually orthogonal planes. Remarkably, this phenomenon fails at higher orders. We establish a multiplicative factorization of the unit tangent map into elementary planar rotations acting on a fixed unit vector, providing a  characterization of equality cases at arbitrary order. As a result, we construct a third order equality curve in $\mathbb R^3$ whose Fourier modes cannot be decomposed into mutually orthogonal circular modes.
\end{abstract}
\maketitle
\section{Introduction}

A fundamental result in classical geometry is the isoperimetric inequality, which asserts that for any simple closed curve in the Euclidean plane with length $L$ and enclosed area $A$, one has $L^2 \geq 4 \pi A$, where equality holds precisely when the curve is a circle. Beyond its classical form, two closely related directions have proved particularly fruitful: extending isoperimetric inequalities to submanifolds of Euclidean space, and sharpening the planar inequality by incorporating more refined geometric information.

The first line of research concerns the geometry of minimal submanifolds. Carleman's work on minimal surfaces established an early connection between minimality and isoperimetry \cite{Carleman1921}; see also Osserman's survey \cite{Osserman1978}. Subsequently, Chakerian \cite{Chakerian1978} obtained a sharpened inequality for curves bounding minimal surfaces, building on a method of Sachs \cite{Sachs1960}. Li, Schoen and Yau \cite{LiSchoenYau1984} proved the sharp isoperimetric inequality for compact minimal surfaces in $\R^n$ with weakly connected boundary (Definition \ref{def:weakly connected}), with equality precisely for flat disks. Their proof exploits the Poincar\'e inequality for the boundary coordinate functions, an approach closely related to the Wirtinger inequalities employed in the present paper. For area-minimizing spanning surfaces, Almgren \cite{Almgren1986} established optimal isoperimetric inequalities. Stone \cite{Stone2003} subsequently derived an integral representation of the isoperimetric deficit for general minimal surfaces, involving geometric quantities on the surface and its boundary, and obtained an improved universal isoperimetric estimate. A major recent development for general minimal submanifolds is due to Brendle \cite{Brendle2021}, who proved a Sobolev inequality for submanifolds in Euclidean space of arbitrary dimension and codimension. In particular, this implies the sharp isoperimetric inequality for minimal submanifolds of codimension at most two. These results relate isoperimetric inequalities to the geometry of the spanning submanifold, particularly its mean curvature.

A closely related line of research seeks sharper inequalities that capture geometric information beyond the length and enclosed area. Such refinements have a long history, including the inequalities of Sachs \cite{Sachs1960}, Chakerian \cite{Chakerian1978}, and the Bonnesen-type estimates reviewed by Osserman \cite{Osserman1978}; see also McGrath's geometric formula for the isoperimetric deficit \cite{McGrath2021}. Lee and the second author \cite{kwong2021higher} developed higher order Wirtinger-type inequalities using Fourier analysis. They obtained a family of sharp isoperimetric inequalities of arbitrarily high order for convex plane curves, as well as a second order refinement applicable to general plane curves. Here ``higher order'' means that, in addition to the perimeter and area, the inequalities involve the curvature of the boundary and its derivatives. The second author \cite{Kwong2024} subsequently developed a unified method giving arbitrary-order isoperimetric inequalities for general smooth plane curves, including sharp upper and lower bounds for the isoperimetric deficit, and extended the same spectral approach to discrete isoperimetric inequalities for polygons.

Motivated by these developments, the aim of this paper is to extend higher order isoperimetric inequalities from plane curves to surfaces in Euclidean space. In particular, we ask whether the higher order planar isoperimetric inequalities obtained using Wirtinger-type inequalities can be established for minimal surfaces in arbitrary codimension. We answer this question affirmatively. In fact, minimality is not required. Our inequalities hold for any smooth compact connected embedded surface $M\subset\R^n$ whose boundary is a smooth simple closed curve. They involve arbitrarily high derivatives of the boundary curve, together with the mean curvature of $M$. We also investigate the equality cases, where minimality leads to additional geometric restrictions and different rigidity phenomena at different orders.
At first order, we also generalize the result of Li, Schoen and Yau \cite{LiSchoenYau1984} for minimal surfaces with weakly connected boundaries; see Theorem \ref{thm:LSY isoperimetric inequality}. Our inequality holds without assuming minimality and provides a refinement involving the boundary conormal and the mean curvature of the surface.

To state the main result, let $A$ be the area of $M$, let $L$ be the length of $C$, and translate the origin to the centroid of $C$. Write $R=\frac{L}{ 2\pi }$, denote by $\nu$ the outward unit conormal to $C$ in $M$, and let $\kappa=\frac{d^2x}{ds^2}$ be the ambient curvature vector of $C$. We use the convention $H=\Delta_Mx$ for the mean curvature vector. Define
\begin{equation*}
\D(C, M)=\frac{1}{\pi}(L^2-4\pi A) -\frac1R\int_C|x-R\nu|^2\, ds -2\int_M x\cdot H\, dA,
\end{equation*}
and for each integer $l\ge1$,
\begin{equation*}
\E_l(C)=R^{2l-3}\int_C
\left|\left(\frac{d}{ds}\right)^{l-1}
(x+R^2\kappa)\right|^2\, ds.
\end{equation*}
Note that $\mathcal{E}_{m-1}(C)$ is an $m$-th order boundary quantity.
For $m\ge2$, introduce the polynomials
\begin{equation*}
P_m(u)=\prod_{j=2}^m(u-j^2), \qquad \frac{P_m(u)-P_m(1)}{u-1} =\sum_{l=1}^{m-1}s_{m, l}u^{l-1}.
\end{equation*}
Our main result is the following hierarchy.

\begin{theorem}[Higher order isoperimetric inequalities]
Under the above assumptions,
$$
\mathcal{D}(C, M)\ge0,
$$
and for every integer $m\ge2$,
\begin{equation*}
P_m(1)\, \D(C, M)+\sum_{l=1}^{m-1}s_{m, l}\E_l(C)\ge0, \qquad P_m(1)=\frac{(-1)^{m-1}}2(m-1)!(m+1)!.
\end{equation*}
Equality in the inequality of order $m$ holds if and only if, under the parameter $t=\frac{s}{R}$, the boundary position vector has the finite Fourier expansion
\begin{equation}\label{eq:intro-fourier}
x(t)=\sum_{k=1}^m\left(C_k\cos kt+D_k\sin kt\right), \qquad C_k, D_k\in\R^n.
\end{equation}
\end{theorem}

For example, the first two nontrivial members of the hierarchy give

\begin{equation*}
\begin{aligned}
0\le \D(C,M)
&=\frac{L^2-4\pi A}{\pi}
-\frac1R\int_C|x-R\nu|^2\,ds
-2\int_M x\cdot H\,dA\\
&\le \frac13\E_1(C)
=\frac{1}{3R}\int_C|x+R^2\kappa|^2\,ds,
\end{aligned}
\end{equation*}
and
\begin{equation*}
\begin{aligned}
\D(C,M)
&\ge \frac12\E_1(C)-\frac1{24}\E_2(C)\\
&=\frac{1}{2R}\int_C|x+R^2\kappa|^2\,ds
-\frac{R}{24}\int_C
\left|\frac{d}{ds}(x+R^2\kappa)\right|^2\,ds.
\end{aligned}
\end{equation*}

Although $\D(C,M)$ involves the conormal and mean curvature of $M$, remarkably it depends only on the boundary curve $C$, since $\D(C,M)=\int_0^{2\pi}(|\dot x|^2-|x|^2)\,dt$ by Lemma \ref{lem:boundary-identity}. Thus, our inequalities hold without assuming minimality, although minimality imposes additional rigidity on the equality cases.

Indeed, the equality cases are the most interesting aspect of our results. The second author  and Lee's higher order Wirtinger inequality \cite{kwong2021higher} singles out finite Fourier expansions, as in \eqref{eq:intro-fourier}. For curves in higher-dimensional Euclidean space, however, the constant-speed condition imposes further algebraic constraints on their Fourier coefficients. At first order, equality forces $C$ to be a round circle; if $M$ is minimal, it then forces $M$ to be a flat disk. At second order we obtain a normal form: up to an ambient orthogonal transformation,
\begin{equation}\label{eq:intro-second-normal-form}
x(t)=a(e_1\cos t+e_2\sin t) +b(e_3\cos2t+e_4\sin2t), \qquad a>0, \quad b\ge0, \quad a^2+4b^2=R^2,
\end{equation}
where $e_1, e_2, e_3, e_4$ are orthonormal whenever $b>0$. Thus the two nonzero circular modes lie in orthogonal two-planes. In particular, if $n\le3$ and $M$ is minimal, equality at second order occurs precisely when $M$ is a flat disk. The restriction on ambient dimension is essential: in $\R^4\cong\C^2$, the holomorphic disk
\begin{equation*}
M=\{(z, z^2): |z|\le1\}, \qquad x(t)=(\cos t, \sin t, \cos2t, \sin2t)
\end{equation*}
is a non-flat minimal surface achieving equality at second order.

At third and higher orders, this orthogonal description no longer suffices. Even though the equality condition still restricts $x$ to finitely many Fourier modes, distinct modes need not lie in mutually orthogonal planes. This phenomenon already occurs for curves in $\R^3$. To describe the equality cases, we consider the unit tangent map $T(t)=\frac{\dot x(t)}{R}$. In the equality case, $T=T(z): \Sph^1\to\Sph^{n-1}$ is a finite Laurent-polynomial loop for $z=e^{it}$. This relates the equality problem to the theory of polynomial loops on spheres \cite{BauerCrabb2004}. We prove that the unit tangent map admits a decomposition into finitely many elementary rotations. More precisely, the equality condition of order $m$ is equivalent to the existence of an integer $1\le r\le m$, oriented two-planes $V_1, \dots, V_r\subset\R^n$, and a fixed unit vector $\tau_0\in\Sph^{n-1}$ such that
\begin{equation}\label{eq:intro-factorization}
\frac{\dot x(t)}{R}= \mathcal R_{V_r}(t)\mathcal R_{V_{r-1}}(t) \cdots\mathcal R_{V_1}(t)\tau_0,
\end{equation}
where $\mathcal R_V(t)$ denotes rotation through angle $-t$ in the oriented plane $V$, acting as the identity on $V^\perp$. The planes $V_j$ are not required to be mutually orthogonal.

An illustration (Example \ref{ex:third-order}) is the curve in $\R^3$ given by
\begin{equation*}
x(t)=\left( 4\sin t, \; -2\cos t+\cos2t+\frac23\cos3t, \; -2\sin t+\sin2t+\frac23\sin3t \right).
\end{equation*}
It is a smooth simple closed curve, has zero centroid, and satisfies $|\dot x(t)|^2\equiv20$. Since its position vector contains precisely the first three Fourier frequencies, it achieves equality in the third order inequality for any smooth embedded surface that it bounds, i.e. $\D=\frac{\E_1}{2}-\frac{\E_2}{24}$. The curve is noncircular and nonplanar, and its three nonzero modes cannot be supported on mutually orthogonal two-planes in $\R^3$. It therefore demonstrates why the higher order rigidity problem requires the multiplicative factorization \eqref{eq:intro-factorization}, rather than a straightforward extension of the orthogonal normal form \eqref{eq:intro-second-normal-form}.

The rest of the paper is organized as follows. Section~\ref{sec:preliminaries} recalls the higher order Wirtinger inequalities and establishes notation. Section \ref{sec:inequalities} establishes the higher order isoperimetric inequalities and their equality conditions, together with first-order extensions to surfaces with multiple boundary components. In Section~\ref{sec:low-rigidity}, we establish rigidity results for the first- and second-order inequalities and provide an example of a non-flat minimal surface in $\mathbb{R}^4$ for which equality holds in the second order case. Finally, Section~\ref{sec:loops} establishes the multiplicative factorization of unit algebraic loops and characterizes all higher order equality cases. We also show that in $\mathbb R^3$, non-flat minimal disks attaining equality exist at every order $m\ge3$.

\section{Preliminaries}\label{sec:preliminaries}

\subsection{Notation}
In this paper, we primarily consider surfaces satisfying the following assumption.

\begin{assumption}\label{ass:A}
Let $M\subset\R^n$, $n\ge2$, be a smooth compact connected embedded surface whose boundary $C=\partial M$ is a smooth simple closed curve. Let $x$ denote the position vector and $s$ the arclength along $C$. Assume that the origin is the centroid of $C$, so that
\begin{equation*}
\int_C x\,ds=0.
\end{equation*}
\end{assumption}

While most of our results are established under \hyperref[ass:A]{Assumption~\ref*{ass:A}}, we also extend the first-order inequality to surfaces with multiple boundaries in Subsection~\ref{subsec:weakly-connected}. The precise assumptions are stated in each result.

Let $A$ be the area of $M$ and $L$ be the length of $C$. The vector $\nu$ denotes the outward unit conormal of $C$ in $M$.
The ambient curvature vector of $C$ is
$$\kappa=\frac{ d^2x}{ d s^2}. $$
We use the following convention for the mean curvature vector:
\begin{equation*}
H=\Delta_M x,
\end{equation*}
where $\Delta_M=\Div_M\nabla_M$.

Set
\begin{equation*}
R=\frac{L}{2\pi}, \qquad t=\frac{s}{R}.
\end{equation*}
By abuse of notation, we also use $x(t)$ for the resulting $2\pi$-periodic parametrization of $C$ and a dot for differentiation with respect to $t$. We also write $x^{(l)}=\frac{d^l}{d t^l}x$. Derivatives with respect to $s$ are written explicitly. Then we have
\begin{equation*}
|\dot x(t)|=R, \qquad \int_0^{2\pi}x(t) d t=0,
\qquad d s=R d t.
\end{equation*}
We recall the following definitions, already introduced in the Introduction.
\begin{definition}\label{def:deficit}
Define
\begin{equation*}
\D(C, M)=\frac{L^2-4\pi A}{\pi} -\frac1R\int_C|x-R\nu|^2 d s -2\int_M x\cdot H d A.
\end{equation*}
For each integer $l\ge1$, define
\begin{equation}\label{eq:curvature-energies}
\E_l(C)=R^{2l-3}\int_C \left|\left(\frac{ d }{ d s}\right)^{l-1} (x+R^2\kappa)\right|^2 d s.
\end{equation}
We may abbreviate these quantities to $\D$ and $\E_l$.
\end{definition}

\subsection{Higher order Wirtinger inequalities}
For $m\ge2$, as in the Introduction,  we recall  
\begin{equation*}
P_m(u)=\prod_{j=2}^{m}(u-j^2),
\end{equation*}
and 
\begin{equation}\label{eq:constant}
P_m(1)=\frac{(-1)^{m-1}}2(m-1)!(m+1)!.
\end{equation}
We also set
\[S_m(u)=\frac{P_m(u)-P_m(1)}{u-1} =\sum_{l=1}^{m-1}s_{m, l}u^{l-1}.\]
In particular,
\begin{equation}\label{eq:first-polynomials}
S_2(u)=1, \quad P_2(1)=-3,
\qquad S_3(u)=u-12, \quad P_3(1)=24.
\end{equation}

The following family of higher order Wirtinger inequalities \cite[Proposition~1(c)]{kwong2021higher} is the main tool in this paper. For reader's convenience, we include a proof and the equality conditions.

\begin{proposition}\label{prop:wirtinger}
Let $f:\R\to\R^n$ be $2\pi$-periodic and have zero mean.
\begin{enumerate}
\item
If $f$ is of class $C^1$, then
\begin{equation}\label{eq:wirtinger-first}
\int_0^{2\pi}(|\dot f|^2-|f|^2) d t\ge0,
\end{equation}
with equality if and only if $f(t)=C_1\cos t+D_1\sin t$ for some $C_1, D_1\in\R^n$.
\item
If $m\ge2$ and $f$ is of class $C^m$, then
\begin{equation}\label{eq:wirtinger-higher}
P_m(1)\int_0^{2\pi}(|\dot f|^2-|f|^2) d t +\sum_{l=1}^{m-1}s_{m, l} \int_0^{2\pi}|f^{(l+1)}+f^{(l-1)}|^2 d t\ge0.
\end{equation}
Equality in \eqref{eq:wirtinger-higher} holds if and only if
\begin{equation*}
f(t)=\sum_{k=1}^{m}(C_k\cos kt+D_k\sin kt), \qquad C_k, D_k\in\R^n.
\end{equation*}
\end{enumerate}
\end{proposition}

\begin{proof}
Write the Fourier expansion as $f(t)=\sum_{k=1}^{\infty}(C_k\cos kt+D_k\sin kt) =\sum_{k=1}^{\infty}f_k(t)$. By Parseval's identity,
\begin{equation}\label{eq 1st order}
\int_0^{2\pi}(|\dot f|^2-|f|^2) d t
=\pi\sum_{k\ge1}(k^2-1)(|C_k|^2+|D_k|^2)\ge0.
\end{equation}
This proves \eqref{eq:wirtinger-first} and its equality condition.

For \eqref{eq:wirtinger-higher}, note that $f_k^{(l+1)}+f_k^{(l-1)}=\left(1-k^2\right) f_k^{(l-1)}$. Hence
$$
\int_0^{2 \pi}\left|f_k^{(l+1)}+f_k^{(l-1)}\right|^2 d t=\pi\left(k^2-1\right)^2 k^{2(l-1)}\left(|C_k|^2+|D_k|^2\right)
$$
and so
$$
\int_0^{2 \pi}\left|f^{(l+1)}+f^{(l-1)}\right|^2 d t=\pi \sum_{k=1}^{\infty}\left(k^2-1\right)^2 k^{2(l-1)}\left(|C_k|^2+|D_k|^2\right).
$$
Combining with \eqref{eq 1st order}, the LHS of \eqref{eq:wirtinger-higher} is
\begin{equation}\label{eq:wirtinger-higher L.H.S.}
\pi \sum_{k=1}^{\infty}\left[P_m(1)\left(k^2-1\right)+\left(k^2-1\right)^2 \sum_{l=1}^{m-1} s_{m, l} k^{2(l-1)}\right]\left(|C_k|^2+|D_k|^2\right).
\end{equation}
By definition,
$\sum_{l=1}^{m-1} s_{m, l} k^{2(l-1)}=S_m\left(k^2\right)$ and $(k^2-1) S_m(k^2)=P_m(k^2)-P_m(1)$. Therefore the coefficient in the square bracket is
\begin{align*}
P_m(1)\left(k^2-1\right)+\left(k^2-1\right)^2 S_m\left(k^2\right) =\left(k^2-1\right) P_m\left(k^2\right)
=\prod_{j=1}^m\left(k^2-j^2\right).
\end{align*}
Putting this into \eqref{eq:wirtinger-higher L.H.S.} gives
\begin{align*}
& P_m(1) \int_0^{2 \pi}\left(|\dot{f}|^2-|f|^2\right) d t+\sum_{l=1}^{m-1} s_{m, l} \int_0^{2 \pi}\left|f^{(l+1)}+f^{(l-1)}\right|^2 d t\\
=& \pi\sum_{k=1}^{\infty}\prod_{j=1}^m\left(k^2-j^2\right) (|C_k|^2+|D_k|^2)\ge0.
\end{align*}
This proves \eqref{eq:wirtinger-higher}, with equality exactly when all Fourier coefficients of order greater than $m$ vanish.
\end{proof}

\section{Higher order isoperimetric inequalities}\label{sec:inequalities}
We first derive analytic expressions for the geometric quantities defined in Definition \ref{def:deficit}.

\begin{lemma}\label{lem:boundary-identity}
For a compact surface $M\subset\R^n$ with (possbily disconnected) smooth boundary $C=\partial M$, we have
\begin{equation}\label{eq:boundary-identity}
\D(C,M)=RL-\frac1R\int_C|x|^2\,ds.
\end{equation}
In particular, it is independent of the spanning surface $M$.

If $C$ is connected, then
\begin{equation}\label{eq:boundary-identity-integral}
\D(C,M)=\int_0^{2\pi}(|\dot x|^2-|x|^2)\,dt
\end{equation}
and for every $l\ge1$,
\begin{equation}\label{eq:derivative-identity}
\E_l(C)=\int_0^{2\pi}|x^{(l+1)}+x^{(l-1)}|^2\,dt.
\end{equation}
\end{lemma}

\begin{proof}
For an ambient vector field $\phi$ along $M$, let
$$
\Div_M\phi=\sum_{i=1}^{2}
\nabla^{\R^n}_{e_i}\phi\cdot e_i,
$$
where $e_1,e_2$ is a local orthonormal frame of $TM$.
Then $\Div_M\phi=\Div_M(\phi^T)-\phi\cdot H$.
Applying this to $\phi=x$ and using the divergence theorem gives
\begin{equation*}
2A=\int_Cx\cdot\nu\,ds-\int_Mx\cdot H\,dA.
\end{equation*}
Expanding $|x-R\nu|^2$ and using $L=2\pi R$, we obtain \eqref{eq:boundary-identity}.

If $C$ is connected, then $|\dot x|=R$ and
$ds=R\,dt$, so
$$
RL-\frac1R\int_C|x|^2\,ds
=\int_0^{2\pi}(|\dot x|^2-|x|^2)\,dt.
$$
This proves \eqref{eq:boundary-identity-integral}.

Finally, $\ddot x=R^2\kappa$ and hence
$$
x^{(l+1)}+x^{(l-1)}
=R^{l-1}\left(\frac{d}{ds}\right)^{l-1}
(x+R^2\kappa).
$$
Compared with \eqref{eq:curvature-energies},
this proves \eqref{eq:derivative-identity}.
\end{proof}

\subsection{Isoperimetric inequalities for connected boundaries}
\begin{theorem}[Higher order isoperimetric inequalities]\label{thm:main}
Suppose $M$ and $C$ satisfy \hyperref[ass:A]{\textbf{Assumption A}}. Then $\D(C, M)\ge0$, or equivalently,
\begin{equation}\label{eq:first-order-expanded}
\frac{L^2-4\pi A}{\pi}
\ge\frac{2\pi}{L}\int_C \left|x-\frac{L}{2\pi}\nu\right|^2 d s +2\int_Mx\cdot H d A.
\end{equation}
For every integer $m\ge2$,
\begin{equation*}
\sum_{l=1}^{m-1}s_{m, l}\E_l(C)+P_m(1)\D(C, M)\ge0.
\end{equation*}
Equivalently,
\begin{equation}\label{eq:hierarchy-expanded}
\begin{aligned}
0\le{}& \sum_{l=1}^{m-1}s_{m, l}R^{2l-3}\int_C \left|\left(\frac{ d }{ d s}\right)^{l-1} (x+R^2\kappa)\right|^2 d s\\
& -\frac{(-1)^m}{2}(m-1)!(m+1)! \left(\frac{L^2-4\pi A}{\pi} -\frac1R\int_C|x-R\nu|^2 d s -2\int_Mx\cdot H d A\right).
\end{aligned}
\end{equation}
Equality in \eqref{eq:first-order-expanded} holds precisely when
$x(t)=C_1\cos t+D_1\sin t$.
For $m\ge2$, equality in \eqref{eq:hierarchy-expanded} holds precisely when
\begin{equation}\label{eq:equality-fourier}
x(t)=\sum_{k=1}^{m}(C_k\cos kt+D_k\sin kt), \qquad C_k, D_k\in\R^n,
\end{equation}
where $t=\frac{2\pi s}{L}$.
\end{theorem}

\begin{proof}
Applying Proposition~\ref{prop:wirtinger} to $x$ and using Lemma~\ref{lem:boundary-identity}, we get \eqref{eq:first-order-expanded}. Formulas \eqref{eq:constant}, \eqref{eq:derivative-identity} and inequality \eqref{eq:wirtinger-higher} give \eqref{eq:hierarchy-expanded}. The equality cases follow from Proposition~\ref{prop:wirtinger}.
\end{proof}

Let us write down explicitly the second and third order case.
\begin{corollary}
Suppose $M$ and $C$ satisfy \hyperref[ass:A]{\textbf{Assumption A}}. Then
\begin{equation}\label{eq:second-order}
0\le\D\le\frac13\E_1 =\frac{2\pi}{3L}\int_C|x+R^2\kappa|^2 d s,
\end{equation}
and
\begin{equation}\label{eq:third-order}
\D\ge\frac12\E_1-\frac1{24}\E_2 =\frac{\pi}{L}\int_C|x+R^2\kappa|^2 d s -\frac{L}{48\pi}\int_C \left|\frac{ d }{ d s}(x+R^2\kappa)\right|^2 d s.
\end{equation}
Equality in the upper bound of \eqref{eq:second-order}, or in \eqref{eq:third-order}, is characterized by \eqref{eq:equality-fourier} with $m=2$, or $m=3$, respectively.
\end{corollary}

\begin{proof}
Substitute \eqref{eq:first-polynomials} into \eqref{eq:hierarchy-expanded}.
\end{proof}

\begin{remark}
Each inequality in the hierarchy provides an upper bound for the nonnegative deficit in the preceding inequality. To make this precise, set
$$
\mathcal B_1=0, \qquad \mathcal B_m=-\frac{1}{P_m(1)} \sum_{l=1}^{m-1}s_{m, l}\E_l \quad (m\ge2).
$$
Since $P_m(1)$ has sign $(-1)^{m-1}$, the inequality of order $m$ can be written as
$$\delta_m:=(-1)^{m-1}(\D-\mathcal B_m)\ge0. $$
Thus $\delta_m$ measures the gap between $\D$ and its order-$m$ bound. The next inequality, $\delta_{m+1}\ge0$, yields
\begin{equation}\label{eq:general-successive-remainder}
0\le\delta_m \le(-1)^{m-1}(\mathcal B_{m+1}-\mathcal B_m).
\end{equation}
Indeed,
$$
\delta_m+\delta_{m+1} =(-1)^{m-1}(\mathcal B_{m+1}-\mathcal B_m).
$$
In particular, equality in the upper bound in \eqref{eq:general-successive-remainder} holds precisely when the inequality of order $m+1$ is an equality.

For example,
$$
\mathcal B_2=\frac13\E_1, \qquad \mathcal B_3=\frac12\E_1-\frac1{24}\E_2,
$$
so the third order inequality gives
$$
0\le\delta_2=\frac13\E_1-\D \le\frac1{24}\E_2-\frac16\E_1.
$$
Substituting $\E_1=3(\D+\delta_2)$ into the right-hand inequality and rearranging gives the equivalent estimate
\begin{equation*}
0\le\frac13\E_1-\D \le\frac1{36}\E_2-\frac13\D.
\end{equation*}
\end{remark}

\begin{remark}
From \eqref{eq:boundary-identity}, $\D$ is determined only by the boundary curve $C$, not the surface $M$, and so is $\E_l$. In particular, equality in Theorem~\ref{thm:main} is independent of the choice of a smooth spanning surface. However, under minimality condition on $M$, we can deduce the rigidity of $M$ in the first order case. See Theorem \ref{thm:first-rigidity}.
\end{remark}
\subsection{Extension to multiple boundaries}\label{subsec:weakly-connected}
We now extend the first order inequality to surfaces with possibly disconnected boundaries, using an argument of Li, Schoen and Yau \cite{LiSchoenYau1984}.

\begin{definition}\label{def:weakly connected}
A compact set $C\subset\mathbb R^n$ is said to be \emph{weakly connected} if there exists an orthogonal coordinate system $(x_1, \ldots, x_n)$ such that, for each $\alpha=1, \ldots, n$, every hyperplane $\{x_\alpha=c\}$ disjoint from $C$ has $C$ entirely on one side.
\end{definition}

In particular, every connected set is weakly connected.

The following theorem generalizes the isoperimetric inequality of Li, Schoen and Yau \cite[Theorem 1]{LiSchoenYau1984}.
\begin{theorem}\label{thm:LSY isoperimetric inequality}
Let $M\subset\mathbb R^n$ be a smooth compact connected embedded surface whose boundary $\partial M$ is weakly connected. Let $L$ be the total length of $\partial M$, and set $R=\frac{L}{2\pi}$. Suppose that the origin is the centroid of $\partial M$. Then
\begin{equation}\label{eq:weakly-connected}
\frac{L^2-4\pi A}{\pi}
\ge
\frac1R\int_{\partial M}|x-R\nu|^2\, ds +2\int_M x\cdot H\, dA.
\end{equation}
Equality holds if and only if $\partial M$ is a circle. In particular, if $M$ is minimal, then
$$
L^2-4\pi A \ge \frac{\pi}{R} \int_{\partial M}|x-R\nu|^2\, ds,
$$
with equality if and only if $M$ is a flat disk.
\end{theorem}

\begin{proof}
By Lemma~\ref{lem:boundary-identity}, we have
\begin{align*}
& \frac{L^2-4\pi A}{\pi} -\frac1R\int_{\partial M}|x-R\nu|^2\, ds -2\int_M x\cdot H\, dA\\
& \qquad=RL-\frac1R\int_{\partial M}|x|^2\, ds.
\end{align*}

By the remarkable translation argument of Li, Schoen and Yau \cite[Eqn. (17)]{LiSchoenYau1984}, weak connectedness and the Poincar\'e inequality on a connected closed Lipschitz curve $\widehat C_\alpha$ (which is obtained by translating each component of $\partial M$ in a special way which preserves $x_\alpha$, cf. \cite[p. 240-241]{LiSchoenYau1984}) implies that
$$
\int_{\partial M} x_\alpha^2\, ds
\le \frac{L^2}{4\pi^2} \int_{\partial M} \left|\frac{dx_\alpha}{ds}\right|^2\, ds
$$
for every $\alpha=1, \ldots, n$. Summing over $\alpha$ gives
$$
\int_{\partial M}|x|^2\, ds \le R^2\int_{\partial M} \left|\frac{dx}{ds}\right|^2\, ds =R^2L.
$$
Substituting this into the preceding identity proves \eqref{eq:weakly-connected}.

We now consider equality. Suppose that $\partial M$ has at least two components $C_1, \ldots, C_p$, with lengths $L_i<L$. Equality would require equality in each of the above Poincar\'e inequalities. Recall that, for each $\alpha$, the translation argument of Li, Schoen and Yau produces a connected closed curve $\widehat C_\alpha$ by translating the components $C_i$ in directions perpendicular to the $x_\alpha$-axis, thereby preserving the coordinate function $x_\alpha$. By the equality condition for the Poincar\'e inequality on $\widehat C_\alpha$, the restriction of $x_\alpha$ to each $C_i$ takes the form
$$
x_\alpha(s)=a_{\alpha,i}\cos\frac{2\pi s}{L}+b_{\alpha,i}\sin\frac{2\pi s}{L},
$$
where $s$ is arclength along $C_i$.
Since $x_\alpha$ is periodic with period $L_i<L$, this forces $x_\alpha\equiv0$ on $C_i$ for such $\alpha$. But $\alpha\in \{1,\cdots,n\}$ is arbitrary, which leads to a contradiction. Thus equality requires $\partial M$ to be connected. This then reduces to Theorem \ref{thm:first-rigidity}.
\end{proof}
We can also extend the first-order inequality to surfaces with arbitrarily many boundary components. The resulting inequality involves the isoperimetric deficit together with an additional term, which can be interpreted as the moment of inertia of the boundary-component centroids about the centroid of $\partial M$, with each component weighted by its length.
\begin{theorem}
Let $M\subset\R^n$, $n\ge2$, be a smooth compact embedded surface, not necessarily connected, whose boundary $C=C_1\cup\cdots\cup C_p$ consists of $p$ smooth simple closed curves. Let $L_i$ and $b_i$ denote the length and arclength centroid of $C_i$, respectively, and set $L=\sum_{i=1}^p L_i$ and $R=\frac{L}{2\pi}$. Assume that the origin is the centroid of $C$. Then
\begin{equation*}
\begin{aligned}
\frac{L^2-4\pi A}{\pi}
+\frac1R\sum_{i=1}^p L_i|b_i|^2
\ge {}&
\frac1R\int_C|x-R\nu|^2\, ds\\
& +2\int_M x\cdot H\, dA
+\frac{L^3-\sum_{i=1}^p L_i^3}{2\pi L}.
\end{aligned}
\end{equation*}
Equality holds if and only if every $C_i$ is a circle.
\end{theorem}

\begin{proof}
By Wirtinger's inequality applied to \(x-b_i\) on each \(C_i\), where
\(\int_{C_i}(x-b_i)\,ds=0\) by definition of \(b_i\),
\begin{align*}
\int_C|x|^2\, ds
 =\sum_{i=1}^p\left( \int_{C_i}|x-b_i|^2\, ds+L_i|b_i|^2 \right)
& \le\sum_{i=1}^p\frac{L_i^2}{4\pi^2} \int_{C_i}|\nabla x|^2\, ds +\sum_{i=1}^pL_i|b_i|^2\\
& =\frac{1}{4 \pi^2} \sum_{i=1}^p L_i^3+\sum_{i=1}^p L_i\left|b_i\right|^2 .
\end{align*}
Substituting this into the identity
\begin{align*}
\frac{L^2-4\pi A}{\pi} -\frac1R\int_C|x-R\nu|^2\, ds -2\int_Mx\cdot H\, dA =RL-\frac1R\int_C|x|^2\, ds
\end{align*}
proves the inequality. Equality holds precisely when equality holds in Wirtinger's inequality on every $C_i$, which is equivalent to each $C_i$ being a circle.
\end{proof}
\section{Equality at the first and second orders}\label{sec:low-rigidity}

\subsection{First order equality }
\begin{theorem} [First order rigidity]\label{thm:first-rigidity}
Equality holds in \eqref{eq:first-order-expanded} if and only if $C$ is a planar circle centered at the origin. If $M$ is minimal, equality holds if and only if $M$ is a flat disk centered at the origin.
\end{theorem}

\begin{proof}
In the equality case, Theorem~\ref{thm:main} gives
$x(t)=C_1\cos t+D_1\sin t$. The constant-speed condition is
$$
R^2=|C_1|^2\sin^2t+|D_1|^2\cos^2t
-2C_1\cdot D_1\sin t\cos t.
$$
Its nonconstant Fourier coefficients vanish, so $|C_1|=|D_1|=R$ and $C_1\cdot D_1=0$. Thus $C$ is a circle of radius $R$. Conversely, its constant-speed parametrization contains only the first Fourier mode and therefore gives equality.

Now assume that $H=0$, and let $V$ be the plane containing $C$. Choose orthonormal coordinates such that $V=\{x_3=\cdots=x_n=0\}$. Since $\Delta_M x=H=0$, each coordinate function $x_\alpha$, $\alpha=3, \ldots, n$, satisfies $\Delta_M x_\alpha=0$ on $M$ and $x_\alpha=0$ on $\partial M=C$. The maximum principle gives $x_\alpha\equiv0$ on $M$, so $M\subset V$ and $\partial M$ is a circle. Therefore $M$ is a flat disk.
\end{proof}

Let us single out the first order isoperimetric inequality in the case where $M$ is minimal.
\begin{corollary}\label{cor:minimal}
If $M$ is minimal and $\partial M$ is connected, then
\begin{equation*}
L^2-4\pi A\ge\frac{2\pi^2}{L}\int_C|x-R\nu|^2 d s\ge0.
\end{equation*}
In particular, $L^2=4\pi A$ if and only if $M$ is a flat disk.
\end{corollary}
This recovers the classical rigidity in Carleman's isoperimetric inequality; the conormal term gives the sharpened form associated with Chakerian. The relation with the planar inequalities of Sachs and Chakerian is discussed in \cite{kwong2021higher}.

\subsection{Second order equality}
At the second order, the constant-speed condition forces the two Fourier modes to lie in orthogonal planes.
\begin{theorem}[Second order equality]\label{thm:second-rigidity}
Equality holds in the upper bound of \eqref{eq:second-order} if and only if, up to an orthogonal transformation, the normalized arclength parametrization has the form
\begin{equation}\label{eq:second-normal-form}
x(t)=a(e_1\cos t+e_2\sin t)
+b(e_3\cos2t+e_4\sin2t),
\end{equation}
where $a>0$, $b\ge0$, and
\begin{equation}\label{eq:second-speed}
a^2+4b^2=R^2.
\end{equation}
If $b>0$, the four vectors $e_1, e_2, e_3, e_4$ are orthonormal. If $b=0$, the second summand is omitted and only $e_1, e_2$ are needed. Consequently, $C$ lies in a linear subspace of dimension at most four, and for $n\le3$ it is a planar circle.
\end{theorem}

\begin{proof}
By Theorem~\ref{thm:main}, equality gives
$$x(t)=C_1\cos t+D_1\sin t+C_2\cos2t+D_2\sin2t. $$
For convenience put $U=-C_1$, $V=D_1$, $W=-2C_2$, and $Z=2D_2$. Then
$$\dot x(t)=U\sin t+V\cos t+W\sin2t+Z\cos2t. $$
The frequency-two and frequency-four coefficients in the identity $|\dot x(t)|^2=R^2$ give, respectively,
$$
|U|=|V|, \quad U\cdot V=0, \qquad |W|=|Z|, \quad W\cdot Z=0.
$$
The mixed terms have frequencies one and three. Their vanishing gives
\begin{align*}
U\cdot W+V\cdot Z& =0, & -U\cdot W+V\cdot Z& =0, \\
-U\cdot Z+V\cdot W& =0, & U\cdot Z+V\cdot W& =0.
\end{align*}
Thus every inner product between $\{U, V\}$ and $\{W, Z\}$ is zero. Equivalently,
\begin{align*}
|C_1|=|D_1|=a, \qquad |C_2|=|D_2|=b, \\
C_1\perp D_1, \qquad C_2\perp D_2, \qquad \Span\{C_1, D_1\}\perp\Span\{C_2, D_2\}.
\end{align*}

Choosing orthonormal bases in the nonzero planes gives \eqref{eq:second-normal-form}, and the constant term in the speed identity gives \eqref{eq:second-speed}.

If $a=0$, the nonconstant curve satisfies $x(t+\pi)=x(t)$ and traverses its image twice on $[0, 2\pi]$, contradicting the simple closed parametrization. Hence $a>0$. If $n\le3$, two orthogonal nonzero two-planes cannot occur, so $b=0$.

Conversely, \eqref{eq:second-normal-form} has constant speed $R$, has zero mean, and contains only the first two Fourier modes. Its projection onto $\Span\{e_1, e_2\}$ is a once-traversed circle, so it is simple. Therefore it satisfies the equality condition in Theorem~\ref{thm:main} for $m=2$.
\end{proof}

\begin{corollary}\label{cor:second-minimal}
If $n=3$ and $M$ is minimal, equality in the upper bound of \eqref{eq:second-order} holds if and only if $M$ is a flat disk.
\end{corollary}

\begin{proof}
Theorem~\ref{thm:second-rigidity} implies $C$ is a planar circle. The argument in the proof of Theorem ~\ref{thm:first-rigidity} then implies $M$ is a flat disk. The converse follows from $\D=\E_1=0$.
\end{proof}

\begin{example}[A non-round equality case on a minimal surface]
In $\R^4\cong\C^2$, consider the graph of the holomorphic function $z^2$:
$$
M=\{(z, z^2):|z|\le1\}, \qquad
x(t)=(\cos t, \sin t, \cos2t, \sin2t).
$$
The parametrization $X(u, v)=(u, v, u^2-v^2, 2uv)$ is harmonic and conformal, with induced metric
$$(1+4(u^2+v^2))(d u^2+ d v^2). $$
It is an embedding, and hence parametrizes a smooth minimal disk. Direct integration gives
$$
R=\sqrt5, \qquad L=2\pi\sqrt5, \qquad A=3\pi, \qquad L^2-4\pi A=8\pi^2.
$$
Along the boundary,
\begin{align*}
\kappa& =-\frac15(\cos t, \sin t, 4\cos2t, 4\sin2t), \\
\nu& =\frac1{\sqrt5}(\cos t, \sin t, 2\cos2t, 2\sin2t).
\end{align*}
Thus $|x-R\nu|^2=1$ and $|x+R^2\kappa|^2=9$, from which
$$\D=8\pi-2\pi=6\pi, \qquad \E_1=18\pi. $$
This verifies equality $\D=\frac{\E_1}{3}$ directly. The curve is not planar, as its two nonzero circular modes span orthogonal two-planes. Thus the dimensional restriction in Corollary~\ref{cor:second-minimal} is essential.
\end{example}

\section{Algebraic loops and higher order equality}\label{sec:loops}

\subsection{Failure of an orthogonal Fourier decomposition}
The first- and second-order classification do not extend to higher order. Although equality still implies a finite Fourier expansion (Theorem~\ref{thm:main}), different Fourier modes need not lie in mutually orthogonal planes, as the following example shows.

\begin{example}[A non-trivial third order equality curve in $\R^3$]\label{ex:third-order}
Consider the $2\pi$-periodic curve
\begin{equation*}
x(t)=\left(4\sin t, \,
-2\cos t+\cos2t+\frac23\cos3t, \,
-2\sin t+\sin2t+\frac23\sin3t\right).
\end{equation*}
Its velocity is
\begin{equation*}
\dot x(t)=2(2\cos t, \,
\sin t-\sin2t-\sin3t, \,
-\cos t+\cos2t+\cos3t).
\end{equation*}
The sum of the squares of the last two components inside the parentheses is
$|-e^{it}+e^{2it}+e^{3it}|^2 =3-2\cos2t=5-4\cos^2t$.
Consequently $|\dot x|^2=20$, so $R=\sqrt{20}$ and $L=4\pi\sqrt5$. The mean of $x$ is zero.

To check simplicity, suppose $x(t)=x(u)$, with parameters $\textrm{mod }2\pi$. The first coordinate gives either $u=t$ or $u=\pi-t$. In the second case, writing $c=\cos t$, the difference of the second coordinates is
$$x_2(t)-x_2(\pi-t)=\frac83c(2c^2-3). $$
Since $|c|\le1$, this can vanish only when $c=0$, in which case $\pi-t=t$ modulo $2\pi$. Thus $x$ is injective on $\R/ 2\pi\mathbb Z $. Therefore $x$ is a smooth embedded closed curve.
The curve also bounds a smooth embedded disk. Indeed, its planar projection $\gamma(t)=(x_1(t), x_2(t))$ is injective by the preceding argument and is regular. Let $D\subset\R^2$ be the domain enclosed by $\gamma$. The boundary height function $h(\gamma(t))=x_3(t)$ extends smoothly to $\overline D$. Its graph is then a smooth embedded disk with boundary $C$.

The curve has nonzero Fourier modes of orders one, two, and three. Since $\R^3$ cannot contain two mutually orthogonal two-dimensional subspaces, these modes cannot be supported in mutually orthogonal two-planes. Thus the orthogonal decomposition obtained at the second order does not extend to the third order.

Every smooth compact spanning surface $M$ with boundary this curve satisfies equality in \eqref{eq:third-order}, by Theorem~\ref{thm:main}. Indeed, direct calculation gives
$$
\D=\frac{118\pi}{9}, \qquad \E_1=\frac{674\pi}{9}, \qquad \E_2=584\pi,
$$
and hence $\D=\frac{\E_1}{2} -\frac{E_2}{24}$. In particular, this equality occurs in $\R^3$ without an orthogonal circular-mode decomposition. Figure~\ref{fig:third-curve} shows the curve.

Indeed, the curve in this example is constructed from the multiplicative
factorization in Theorem \ref{thm:general-rigidity}. Let
$$
V_1=\Span\{e_3, e_2\}, \qquad V_2=\Span\{e_1, e_2\},
$$
with orientations given by the ordered bases, and set $\tau=\frac{1}{\sqrt5}(2e_1+e_3)$.
Then its unit tangent $T(t)$ is obtained by
$$ T(t)=\mathcal R_{V_1}(t)^2\mathcal R_{V_2}(t)\tau =\mathcal R_{V_1}(2t)\mathcal R_{V_2}(t)\tau. $$
\end{example}
\begin{figure}[htbp]
\centering
\begin{tikzpicture}[scale=0.7]
\begin{axis}[
width=10cm, height=7cm,
view={35}{25},
xlabel={$x_1$}, ylabel={$x_2$}, zlabel={$x_3$},
tick label style={font=\small},
label style={font=\small},
axis lines=box,
grid=major,
grid style={black!15},
xtick={-4,0,4},
ytick={-2,0,2},
ztick={-3,0,3},
colormap/viridis,
enlargelimits=0.08
]
\addplot3[
very thick,
samples=401,
domain=0:360,
samples y=0,
point meta=z,
shader=interp,
mesh
]
({4*sin(x)},
{-2*cos(x)+cos(2*x)+(2/3)*cos(3*x)},
{-2*sin(x)+sin(2*x)+(2/3)*sin(3*x)});
\end{axis}
\end{tikzpicture}
\caption{The curve in Example~\ref{ex:third-order}.}
\label{fig:third-curve}
\end{figure}
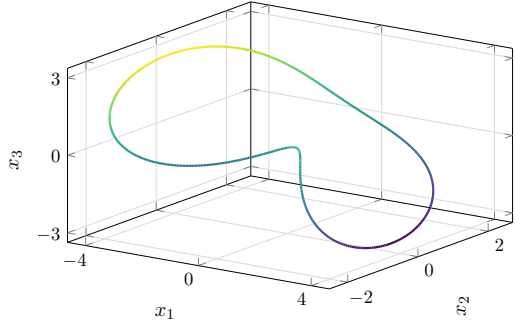

\subsection{Unit algebraic loops }

The equality cases can be studied through the theory of algebraic loops on spheres. To explain this connection, suppose that equality holds in the inequality of order $m$. The equality condition in the higher order Wirtinger inequality gives
$$
x(t)=\sum_{k=1}^{m}
\left(C_k\cos kt+D_k\sin kt\right),
\qquad C_k, D_k\in\R^n.
$$
Differentiating shows that $\dot x(t)$ is also a trigonometric polynomial of degree at most $m$. Moreover, since $t=\frac{2\pi s}{L}$, we have
$$|\dot x(t)|=\frac{L}{2\pi}=R. $$
Thus the normalized tangent
$$T(t):=\frac1R\dot x(t)$$
satisfies $|T(t)|=1$ and has a finite Fourier expansion. In complex notation, we may write
$$
T(t)=\sum_{k=-m}^{m}A_k e^{ikt}
=\sum_{k=-m}^{m}A_k z^k,
\qquad z=e^{it}, \qquad A_k\in\C^n.
$$
Because $T$ is real-valued, its coefficients satisfy $A_{-k}=\overline{A_k}$. Also, $A_0=0$ because $x$ is closed.

Consequently, $T$ defines a map from $\Sph^1$ to $\Sph^{n-1}$ whose coordinate functions are finite Fourier series. Equivalently, regarding $z=e^{i t}$ as the coordinate on $\mathbb{S}^1 \subset \mathbb{C}$, the map $T$ may be viewed as a finite Laurent polynomial with values in $\mathbb{C}^n$. We shall refer to such maps as \emph{algebraic loops on the sphere}.

We shall show that every such loop can be expressed as a product of rotations applied to a fixed unit vector. Each rotation acts in an oriented two-dimensional plane and fixes its orthogonal complement. Applied to the normalized tangent $T$, this provides a description of the equality curves at every order.

This approach is related to the theory of polynomial loops on spheres developed in \cite{BauerCrabb2004}. We give a self-contained proof of the factorization needed here.

\begin{definition}
A \emph{unit algebraic loop} is a map
$T:\R/ 2\pi\mathbb Z \to\Sph^{n-1}$ of the form
\begin{equation*}
T(t)=T(z)=\sum_{k=-r}^{r}A_k z^k,
\qquad z=e^{it}, \qquad A_k\in\C^n,
\end{equation*}
where $A_{-k}=\overline{A_k}$ and $|T(t)|=1$. Its degree is the largest $r$ with $A_r\ne0$.
\end{definition}

On $\C^n$ we define the complex bilinear form $B$ and the Hermitian inner product $\langle\cdot, \cdot\rangle$ by
\begin{equation*}
B(U, W)=\sum_{\alpha=1}^{n}U_\alpha W_\alpha, \qquad \langle U, W\rangle=\sum_{\alpha=1}^{n}U_\alpha\overline{W_\alpha}.
\end{equation*}

Since $T(z)\in\R^n$ is real-valued for $|z|=1$, the bilinear form agrees with the squared Euclidean norm on the unit circle:
$$
B(T(z), T(z)) =\sum_{\alpha=1}^{n}T_\alpha(z)^2 =|T(z)|^2=1, \qquad |z|=1.
$$
Moreover, $B(T(z), T(z))-1$ is a Laurent polynomial that vanishes on the unit circle, so it vanishes identically on $\mathbb C$. Comparing coefficients in
$$
B(T(z), T(z))
=\sum_{j=-2r}^{2r}z^j
\sum_{p+q=j}B(A_p, A_q)=1
$$
therefore gives
\begin{equation*}
\sum_{p+q=j}B(A_p, A_q)=
\begin{cases}1, & j=0, \\
0, & j\ne0. \end{cases}
\end{equation*}

\begin{lemma}\label{lem:highest-coefficients}
If a unit algebraic loop has degree $r>0$, then
\begin{equation}\label{eq:highest-relations}
B(A_r, A_r)=0, \qquad B(A_r, A_{r-1})=0.
\end{equation}
\end{lemma}
\begin{proof}
The coefficient of $z^{2r}$ in $B(T, T)$ is $B(A_r, A_r)$. Since $B(T, T)=1$, this coefficient must vanish. Hence
$$B(A_r, A_r)=0. $$
Similarly, the coefficient of $z^{2r-1}$ is
$$
B(A_r, A_{r-1})+B(A_{r-1}, A_r) = 2B(A_r, A_{r-1}),
$$
which must also vanish. Hence
$$B(A_r, A_{r-1})=0. $$
\end{proof}

\begin{definition}[Isotropic lines and rotation factors]
A complex line $\ell\subset\C^n$ is \emph{isotropic} if $B(w, w)=0$ for some (and hence for all) nonzero $w\in\ell$. Writing $w=u+iv$ with $u, v\in\R^n$, isotropy is equivalent to $0=B(w, w)=|u|^2-|v|^2+2 i u \cdot v$, and hence
$$|u|=|v|, \qquad u\cdot v=0. $$
Thus $\ell$ determines a real two-plane $V_\ell=\Span\{u, v\}$ with the orientation of the ordered pair $(u, v)$. Multiplication of $w$ by a nonzero complex scalar preserves this orientation. Conversely, an oriented plane with orthonormal basis $(e, f)$ determines the isotropic line $\C(e+if)$.

Let $\bar\ell$ be the conjugate line and let $N_\ell=(\ell\oplus\bar\ell)^\perp$, where orthogonality is Hermitian. Isotropy implies that $\ell\perp\bar\ell$, so
$$
\C^n=\ell\oplus N_\ell\oplus\bar\ell
$$
is an orthogonal decomposition. Denote the corresponding Hermitian projections by $\pi_\ell$, $\pi_{N_\ell}$, and $\pi_{\bar\ell}$. Define
\begin{equation*}
\sigma_\ell(t)=z\pi_\ell+\pi_{N_\ell}+z^{-1}\pi_{\bar\ell}, \qquad z=e^{it}.
\end{equation*}
\end{definition}

\begin{lemma}\label{lem:rotation}
The operator $\sigma_\ell(t)$ preserves $\R^n$ and restricts there to the rotation through angle $-t$ in the oriented plane $V_\ell$, extended by the identity on $V_\ell^\perp$. In particular, it is real orthogonal and
\begin{equation*}
\sigma_{\bar\ell}(t)=\sigma_\ell(t)^{-1}.
\end{equation*}
\end{lemma}

\begin{proof}
Choose an ordered orthonormal basis $(e, f)$ of $V_\ell$ with
$w=e+if\in\ell$. Since $\sigma_\ell w=zw$ and $\sigma_\ell\bar w=z^{-1}\bar w$, we obtain
$$
\sigma_\ell(t)e=\cos (t)\, e-\sin (t)\, f, \qquad \sigma_\ell(t)f=\sin (t)\, e+\cos (t)\, f.
$$
The operator fixes the real orthogonal complement of $V_\ell$. This gives the stated rotation. Since $V_{\bar\ell}=V_{\ell}$ with reversed orientation, $\sigma_{\bar{\ell}}(t)=\sigma_{\ell}(t)^{-1}$.
\end{proof}

\subsection{ Factorization of unit algebraic loops}
The following is the key degree-reduction lemma.
\begin{lemma}\label{lem:degree-reduction}
Let $T(z)=\sum_{k=-r}^{r}z^kA_k$ be a unit algebraic loop of degree $r>0$, and let $\ell_r=\C A_r$. Then
\begin{equation*}
\widetilde T(z)=\sigma_{\bar\ell_r}(t)T(z)
\end{equation*}
is a unit algebraic loop of degree exactly $r-1$.
\end{lemma}

\begin{proof}
Lemma~\ref{lem:highest-coefficients} implies $\ell_r$ is isotropic. Lemma~\ref{lem:rotation} shows that $\widetilde T$ is real and unit-valued.
Since
$$
\sigma_{\bar\ell_r}(t)
=z\pi_{\bar\ell_r}+\pi_{N_{\ell_r}}+z^{-1}\pi_{\ell_r},
$$
the possible coefficients of degrees $r+1$ and $r$ in $\widetilde T$ are
$$
\pi_{\bar\ell_r}A_r \quad \textrm{ and }\quad \pi_{N_{\ell_r}}A_r+\pi_{\bar\ell_r}A_{r-1} \text{ respectively}.
$$
The first vanishes because $A_r\in\ell_r$. In the second, the first term vanishes for the same reason, and the second vanishes because by \eqref{eq:highest-relations},
$$
\langle A_{r-1}, \overline{A_r}\rangle
=B(A_{r-1}, A_r)=0.
$$
Hence $\widetilde T$ has degree at most $r-1$.

The $(r-1)$-th degree coefficient is
$$
\pi_{\bar\ell_r}A_{r-2} +\pi_{N_{\ell_r}}A_{r-1}+\pi_{\ell_r}A_r
=\pi_{\bar\ell_r}A_{r-2} +\pi_{N_{\ell_r}}A_{r-1}+A_r.
$$
These three summands lie in mutually orthogonal subspaces, and the last is $A_r\ne0$. They cannot cancel and so the degree is exactly $r-1$.
\end{proof}

\begin{proposition}\label{prop:factorization}
Every unit algebraic loop of degree $r$ admits a representation
\begin{equation}\label{eq:factorization}
T(t)=\sigma_{\ell_r}(t)\sigma_{\ell_{r-1}}(t)
\cdots\sigma_{\ell_1}(t)\tau_0,
\end{equation}
where $\ell_1, \ldots, \ell_r$ are isotropic complex lines and $\tau_0\in\Sph^{n-1}$. For $r=0$, the product is the identity. Conversely, any product of this form is a unit algebraic loop of degree at most $r$.
\end{proposition}

\begin{proof}
For degree zero, take $\tau_0=T$. For $r>0$, choose $\ell_r=\C A_r$ and apply
Lemma~\ref{lem:degree-reduction}. By induction, $$
\sigma_{\bar\ell_r}(t)T(t)
=\sigma_{\ell_{r-1}}(t)\cdots\sigma_{\ell_1}(t)\tau_0.
$$
Multiplication by $\sigma_{\ell_r}(t)$ proves \eqref{eq:factorization}. Conversely, each factor is real orthogonal on $\R^n$ and has Fourier degree one. Their product applied to $\tau_0$ is therefore real, unit-valued, and of degree at most $r$.
\end{proof}

\subsection{ Higher order equality cases}
For an oriented real two-plane $V$, write $\mathcal R_V(t)$ for the rotation through angle $-t$ in $V$, acting as the identity on $V^\perp$. By Lemma~\ref{lem:rotation}, this is the real form of $\sigma_\ell(t)$ when $V=V_\ell$.

\begin{theorem}\label{thm:general-rigidity}
Fix $m\ge2$, and let $M$ and $C$ satisfy \hyperref[ass:A]{\textbf{Assumption A}}. The following are equivalent:
\begin{enumerate}
\item Equality holds in \eqref{eq:hierarchy-expanded}.
\item The unit tangent $T(t)=\frac{\dot x(t)}{R}$ is a unit algebraic loop of
degree at most $m$.
\item There are an integer $1\le r\le m$, oriented two-planes
$V_1, \ldots, V_r\subset\R^n$, and $\tau_0\in\Sph^{n-1}$ such that
\begin{equation}\label{eq:tangent-factorization}
\frac{\dot x(t)}{R}= \mathcal R_{V_r}(t)\mathcal R_{V_{r-1}}(t)
\cdots\mathcal R_{V_1}(t)\tau_0.
\end{equation}
\end{enumerate}
\end{theorem}

\begin{remark}
Note that the planes $V_i$ in the above decomposition need not be mutually orthogonal. In fact, higher Fourier modes can be generated by repeatedly applying the same elementary rotation: for any oriented two-plane $V$ and positive integer $k$,
$$\mathcal R_V(t)^k=\mathcal R_V(kt). $$
Thus, repeated rotations in the same plane generate higher Fourier modes, while compositions of rotations in different planes produce non-trivial interactions between Fourier modes, such as the one in Example \ref{ex:third-order}.
\end{remark}

\begin{proof}
If equality holds, \eqref{eq:equality-fourier} shows that $\dot x$ has degree at most $m$. Since $|\dot x|=R$, its normalization $T$ is a unit algebraic loop. It cannot be constant: a nonzero constant velocity cannot parametrize a closed curve. Thus its degree $r$ satisfies $1\le r\le m$. Proposition~\ref{prop:factorization}, followed by Lemma~\ref{lem:rotation}, proves \eqref{eq:tangent-factorization}.

Conversely, the product in \eqref{eq:tangent-factorization} has Fourier degree at most $r\le m$. Closure of the curve gives
\begin{equation*}
\int_0^{2\pi}\dot x(t) d t=0,
\end{equation*}
so the zero-frequency coefficient of $\dot x$ vanishes. Integrating its Fourier series therefore gives a trigonometric polynomial of degree at most $m$ with zero constant term, thanks to the centroid condition. Thus \eqref{eq:equality-fourier} holds, which gives equality in \eqref{eq:hierarchy-expanded}. The same integration argument proves directly that the second condition implies the first.
\end{proof}

\subsection{Non-flat minimal equality cases in $\mathbb R^3$ at higher orders}
Let us revisit the equality cases for higher order $m$, under the additional assumption that $M$ is minimal in $\R^3$. As shown in Corollaries~\ref{cor:minimal} and \ref{cor:second-minimal}, equality at the first or second order forces $M$ to be a flat disk. Example~\ref{ex:third-order} already shows that this boundary rigidity fails for $m=3$. We now strengthen this observation by showing that non-flat minimal disks attaining equality exist at every order $m\ge 3$.

\begin{theorem}
For every integer $m\ge 3$, there exists a smooth embedded non-flat minimal disk in $\R^3$ such that equality holds in
\eqref{eq:hierarchy-expanded}, i.e.
\begin{equation*}
\begin{aligned}
&\sum_{l=1}^{m-1} s_{m,l} R^{2l-3}
\int_C \left|
\left(\frac{d}{ds}\right)^{l-1}
(x+R^2\kappa)\right|^2 ds\\
&\quad-\frac{(-1)^m}{2}(m-1)!(m+1)!
\left(\frac{L^2-4\pi A}{\pi}
-\frac{1}{R}\int_C|x-R\nu|^2 ds\right)
=0.
\end{aligned}
\end{equation*}

\end{theorem}

\begin{proof}
Fix an integer $m\ge 3$ and put $k=m-1$. Let $e_i$ be the standard orthonormal basis of $\R^3$, and consider the oriented
two-planes
$$
V=\Span\{e_2, e_3\},
\qquad
W=\Span\{e_1, e_2\},
$$
with orientations determined by the displayed
ordered bases. For $\varepsilon>0$, set
$$
\tau_\varepsilon =\frac{e_1+\varepsilon e_2} {\sqrt{1+\varepsilon^2}}.
$$
Motivated by Theorem~\ref{thm:general-rigidity}, define
$$
T_\varepsilon(t) =\mathcal R_W(t)\mathcal R_V(t)^k\tau_\varepsilon =\mathcal R_W(t)\mathcal R_V(kt)\tau_\varepsilon\in \mathbb S^2.
$$
A direct computation gives
$$
T_\varepsilon(t)
=\frac{1}{\sqrt{1+\varepsilon^2}}
\begin{pmatrix}
\cos t+\varepsilon\sin t\cos kt\\
-\sin t+\varepsilon\cos t\cos kt\\
-\varepsilon\sin kt
\end{pmatrix}.
$$
Since $k\ge 2$, the Fourier expansion of $T_\varepsilon$ has no constant term.
Consequently, it can be integrated to a $2\pi$-periodic curve $x_\varepsilon(t)$ with zero centroid and
$|\dot x_\varepsilon(t)|=1$.

Explicitly, integration yields
$$
x_\varepsilon(t)
=\frac{1}{\sqrt{1+\varepsilon^2}}
\begin{pmatrix}
\displaystyle
\sin t+\frac{\varepsilon}{2}
\left(
\frac{\cos((k-1)t)}{k-1}
-\frac{\cos((k+1)t)}{k+1}
\right)
\\[6pt]
\displaystyle
\cos t+\frac{\varepsilon}{2}
\left(
\frac{\sin((k-1)t)}{k-1}
+\frac{\sin((k+1)t)}{k+1}
\right)
\\[6pt]
\displaystyle
\frac{\varepsilon}{k}\cos kt
\end{pmatrix}.
$$
For $\varepsilon\ne 0$, this curve has Fourier degree exactly $k+1=m$. Thus it satisfies the equality condition for the inequality of order $m$.

For sufficiently small $\varepsilon>0$, the  projection of $x_\varepsilon$ onto $W$ is a smooth strictly convex simple closed curve. By the classical solvability of the minimal surface Dirichlet problem over convex domains \cite{JenkinsSerrin1966}, $x_\varepsilon$ bounds a smooth embedded minimal graph $M_\varepsilon$ in $\R^3$. Moreover, $M_\varepsilon$ is non-flat, since $x_\varepsilon$ is nonplanar: $x_3$ has a nonzero Fourier mode of order $k$, which does not occur in $x_1,x_2$.
\end{proof}

\end{document}